\documentclass[12pt,leqno,letterpaper]{article}

\usepackage{amsmath,amsthm,enumerate,amssymb}
\usepackage[utf8]{inputenc}
\usepackage[T1]{fontenc}
\usepackage[english]{babel}
\usepackage{graphicx}

\newtheorem{theorem}{Theorem}[section]

\newtheorem{lemma}[theorem]{Lemma}
\newtheorem{corollary}[theorem]{Corollary}
\newtheorem{definition}[theorem]{Definition}

\newtheorem{remark}[theorem]{Remark}

\newtheorem*{acknowledgement*}{Acknowledgement}

\DeclareMathOperator{\Tf}{T}
\newcommand{\T}[1]{\Tf\hskip -0.25em{#1}}

          \newcommand{\df}[2]{\frac{d#1}{d#2}}
          \newcommand{\dto}[2]{\frac{d^2#1}{d#2^2}}
\begin{document}

\title{The fast track to Löwner's theorem}
\author{Frank Hansen}
\date{December 1st 2011\\
{\small First revision July 20 2012}
\\{\small Second revision November 15 2012}
\\{\small Third revision January 20 2013}}

\maketitle

\begin{abstract}

The operator monotone functions defined in the positive half-line are of particular importance. We give a version of the theory in which integral representations for these functions can be established directly without invoking Löwner's detailed analysis of matrix monotone functions of a fixed order or the theory of analytic functions.

We found a canonical relationship between positive and arbitrary operator monotone functions defined in the positive half-line, and this result effectively reduces the theory to the case of positive functions. 
\\[1ex]
{\bf MSC2010} classification: 26A48; 26A51; 47A63\\[1ex]

\end{abstract}

\section{Introduction  and preliminaries}

The functional calculus is defined by the spectral theorem. Since we only deal with matrices the function $ f(x) $ of a hermitian matrix $ x $ is defined for any function $ f $ defined on the spectrum of $ x. $

\begin{definition}
Let $I$ be an interval of any type. A function $ f\colon I \to\mathbf R $ is said to be \textit{n-matrix monotone} (or just \textit{n-monotone}) if
\[
x\le y\quad\Rightarrow\quad f(x)\leq f(y)
\]
for every pair of $ n\times n $ hermitian matrices $ x $ and $y$ with spectra in $I.$
\end{definition}

\begin{definition}
Let $I$ be an interval of any type.
A function $ f\colon I \to\mathbf R $ is said to be \textit{n-matrix convex} (or just \textit{n-convex}) if
\[
f(\lambda x+(1-\lambda)y)\le\lambda f(x)+(1-\lambda)f(y)
\]
for every $ \lambda\in[0,1] $ and every pair of $ n\times n $ hermitian matrices $ x $ and $y$ with spectra in $I$.
\end{definition}
Note that the spectrum of the matrix $ \lambda x+(1-\lambda)y $ in the definition automatically is contained in $ I. $ The functional calculus on the left-hand side is therefore well-defined.

We realise that a point-wise limit of $ n $-monotone ($ n $-convex) functions is $ n $-monotone ($ n $-convex).

\begin{definition}
A function $ f\colon I \to\mathbf R $ defined in an interval $ I $ is said to be operator monotone (operator convex) if it is $ n $-monotone ($ n $-convex) for all natural numbers $ n. $
\end{definition}

We realise that a point-wise limit of operator monotone (operator convex) functions is operator monotone (operator convex).

\subsection{Other proofs of Löwner's theorem}

Karl Löwner\footnote{Karel Löwner was a Czech known under his German name Karl Löwner. Fleeing the Nazis in 1939 he moved to the United States and changed his name to Charles Loewner.} \cite{kn:loewner:1934} analyses in great detail matrix monotone functions of a fixed order and then arrive at the characterisation of operator monotone functions by means of interpolation theory. 

Wigner and von Neumann gives in \cite{kn:wigner:1954} a new proof of Löwner's theorem based on continued fractions which is almost never cited.

Bendat and Sherman \cite{kn:bendat:1955} gives a new proof of  Löwner's theorem that relies on Löwner's detailed analysis of matrix monotone functions of a fixed order but combines it with the Hamburger moment problem. They also rely on Kraus \cite{kn:kraus:1936} to essentially prove that a function is operator convex if and only if the secant-slope function is operator monotone. 

Kor{\'a}nyi \cite{kn:koranyi:1956} gives a new proof of Löwner's theorem by using a variant of Löwner's characterisation of matrix monotone functions of a fixed order and spectral theory for unbounded self-adjoint operators.

The monograph of Donoghue \cite{kn:donoghue:1974} follows \cite{kn:loewner:1934}  closely but introduces some simplifications.

Sparr \cite{kn:sparr:1980} gives a new proof of Löwner's theorem that combines Löwner's characterisation of matrix monotone functions of a fixed order with the theory of interpolation spaces.

The paper \cite{kn:hansen:1982} by Pedersen and the author introduces the idea of first determining the extreme operator monotone functions and then obtain Löwner's theorem by applying Krein-Milman's theorem. The paper does not rely on Löwner's detailed analysis of matrix monotone functions but uses algebraic methods based on Jensen's operator inequality.

The proof in \cite{kn:hansen:1982} is used in a number of other sources including the book of Bhatia \cite{kn:bhatia:1997}.

Ameur \cite{kn:ameur:2003} combines the techniques of applying Jensen's operator inequality as in  \cite{kn:hansen:1982} with interpolation theory in the sense of Foiaş-Lions to obtain a new proof of Löwner's theorem.

    \section{Matrix monotonicity and matrix concavity}

    There is a striking connection between matrix monotonicity and matrix concavity for functions defined in an interval extending to plus infinity.

    \begin{theorem}\label{theorem: 2n-monotone function is n-concave}
    Let $ f:(0,\infty)\to{\mathbf R} $ be a $ 2n $-monotone function where $ n\ge 1. $ Then $ f $ is matrix concave of order $ n. $ In particular, $ f $ is continuous.
    \end{theorem}

    \begin{proof}
	Let $ x_1,x_2 $ be positive definite matrices of order $ n $ and take
$ s\in[0,1]. $ We consider the unitary block matrix $ V $ of order $ 2n\times 2n $ given by
    \[
	V=\left(\begin{array}{cc}
	   	    s^{1/2}     & -(1-s)^{1/2}\\[0.5ex]
			(1-s)^{1/2} & s^{1/2}
	\end{array}\right)
	\]
and obtain by an elementary calculation that
\[
V^*\left(\begin{array}{cc}
			    x_1 & 0\\
			    0   & x_2
			    \end{array}\right)V
=\left(\begin{array}{cc}
                s x_1+(1-s)x_2  & s^{1/2}(1-s)^{1/2}(x_2-x_1)\\[0.5ex]
				s^{1/2}(1-s)^{1/2}(x_2-x_1) & (1-s)x_1+s x_2
					                              \end{array}\right).
\]
We set $ d=-s^{1/2}(1-s)^{1/2}(x_2-x_1) $ and notice that to a given $ \varepsilon>0 $ the difference
\[
\begin{array}{l}
	\left(\begin{array}{cc}
		    s x_1+(1-s)x_2+\varepsilon & 0\\[0.5ex]
		    0                           & 2\lambda
		    \end{array}\right)-
V^*\left(\begin{array}{cc}
			    x_1 & 0\\[0.5ex]
			    0   & x_2
			    \end{array}\right)V
\\[3ex]
\ge\left(\begin{array}{cc}
      \varepsilon  & d\\[0.5ex]
	  d            & \lambda
					                              \end{array}\right)
\qquad\qquad\text{for}\quad\displaystyle  \lambda\ge (1-s)x_1+s x_2.
\end{array}
\]
Since the last block matrix is positive semi-definite for $ \lambda\ge \varepsilon^{-1}\|d\|^2 $ we realize that
\[
	V^*\left(\begin{array}{cc}
			    x_1 & 0\\
			    0   & x_2
			    \end{array}\right)V\le\left(\begin{array}{cc}
		    s x_1+(1-s)x_2+\varepsilon & 0\\[0.5ex]
		    0                           & 2\lambda
		    \end{array}\right)
\]
for a sufficiently large $ \lambda>0. $ Since $ f $ is $ 2n $-monotone
we then obtain
\[
f\left(V^*\left(\begin{array}{cc}
				x_1 & 0\\[0.5ex]
				0   & x_2
				\end{array}\right)V\right)
\le\left(\begin{array}{cc}
f\left(sx_1+(1-s)x_2+\varepsilon\right) & 0\\[0.5ex]
0                              & f(2\lambda)
\end{array}\right)
\]
for such $ \lambda, $ and since
\[
\begin{array}{l}
f\left(V^*\left(\begin{array}{cc}
				x_1 & 0\\[0.5ex]
				0   & x_2
				\end{array}\right)V\right)
=V^*\left(\begin{array}{cc}
f(x_1) & 0\\[0.5ex]
0      & f(x_2)
\end{array}\right)V\\[3ex]
=\left(\begin{array}{cc}
        s f(x_1)+(1-s)f(x_2) & s^{1/2}(1-s)^{1/2}(f(x_2)-f(x_1))\\[0.5ex]
	    s^{1/2}(1-s)^{1/2}(f(x_2)-f(x_1)) & (1-s)f(x_1)+ s f(x_2)
	                 \end{array}\right)
\end{array}
\]
we realize that
\begin{equation}\label{epsilon concave function}
s f(x_1)+(1-s) f(x_2)\le f\left(s x_1+(1-s) x_2+\varepsilon\right).
\end{equation}
Since $ f $ is monotone the right limit $ f^+ $ defined by setting
\[
f^+(t)=\displaystyle\lim_{\varepsilon\searrow 0} f(t+\varepsilon)\qquad t>0
\]
is well-defined. For positive numbers $ t_1,t_2>0 $ we obtain
\[
\begin{array}{rl}
sf^+(t_1)+(1-s)f^+(t_2)&\le sf(t_1+\varepsilon)+(1-s)f(t_2+\varepsilon)\\[1ex]
&\le f\left(st_1+(1-s)t_2+2\varepsilon\right),
\end{array}
\]
where the first inequality follows from the definition of the right limit and the second follows from inequality (\ref{epsilon concave function}) by setting $ x_1=t_1+\varepsilon $ and $ x_2=t_2+\varepsilon. $
By letting $ \varepsilon $ tend to zero we then obtain
\[
sf^+(t_1)+(1-s)f^+(t_2)\le f^+\left(s t_1+(1-s)t_2\right),
\]
therefore $ f^+ $ is concave and thus continuous. Since $ f $ is monotone increasing we have
\[
f^+(t-\varepsilon)\le f(t)\le f^+(t)\qquad t>0,\, 0<\varepsilon<t,
\]
and since $ f^+ $ is continuous we obtain $ f=f^+ $ by letting $ \varepsilon $ tend to zero. Finally, since we established that $ f $ is continuous, we may let $ \varepsilon $ tend to zero in inequality (\ref{epsilon concave function}) to obtain
\[
sf(x_1)+(1-s)f(x_2)\le f\left(s x_1+(1-s)x_2\right),
\]
showing that $ f $ is $ n $-concave.
\end{proof}

The above theorem, with the added condition that $ f $ is continuous, was proved by Mathias \cite{kn:mathias:1991}. That a $ 4n $-monotone function defined in the positive half-line is $ n $-concave already  follows from \cite[proofs of 2.5.~Theorem and 2.1.~Theorem]{kn:hansen:1982}. The idea of the above proof is taken from \cite{kn:hansen:2003:1}.

\begin{corollary}\label{operator monotonicity implies operator concavity}
An operator monotone function $ f:(0,\infty)\to{\mathbf R} $ is automatically operator concave.
\end{corollary}

It is essential for the above result that the function is defined in an interval stretching out to infinity. Without this assumption there are easy counter examples.

\begin{theorem}

Let $ f:(0,\infty)\to\mathbf R $ be a non-negative function which is $ n $-concave for some $ n\ge 1. $ Then $ f  $ is also $ n $-monotone.
\end{theorem}

\begin{proof}
Let $ x $ and $ y $ be positive definite $ n\times n $ matrices with $ x< y $ and take $ \lambda $ in the open interval $ (0,1). $ We may write
\[
\lambda y=\lambda x +(1-\lambda )\bigl(\lambda(1-\lambda)^{-1}(y-x)\bigr)
\]
as a convex combination of two positive definite matrices. Since $ f $ is $ n $-concave we thus obtain
\[
f(\lambda y)\ge\lambda f(x)+(1-\lambda)f(\lambda(1-\lambda)^{-1}(y-x))\ge \lambda f(x),
\]
where we used that $ f $ is non-negative. Since $ f $ is continuous we obtain $ f(x)\le f(y) $ be letting $ \lambda\to 1. $ In the general case, where just $ x\le y, $ we have
\[
\mu x<x\le y\quad\text{for }\,  0<\mu<1, 
\]
since $ x $ is positive definite, and then obtain $ f(\mu x)\le f(y). $ The assertion now follows by letting $ \mu\to 1. $
\end{proof}

The above proof is taken from \cite[2.5.~Theorem]{kn:hansen:1982}.

\begin{corollary}\label{monotonicity and concavity}
A function mapping the positive half-line into itself is operator monotone if and only if it is operator concave.
\end{corollary}

\subsection{Regularization}

The following regularization procedure is standard, cf. for example
    \cite[Page 11]{kn:donoghue:1974}.
    Let $ \varphi $ be a positive and even $ C^{\infty} $-function defined in the real line, vanishing
    outside the closed interval $ [-1,1] $ and normalized such that
    \[
    \int_{-1}^1 \varphi(x)\,dx=1.
    \]
    For any locally integrable function $ f $ defined in an open interval $ (a,b), $ where possibly $ b=\infty, $ we form, for small $ \varepsilon>0, $ its regularization,
    \[
    f_\varepsilon(t)=\frac{1}{\varepsilon}\int_a^b \varphi\left(\frac{t-s}{\varepsilon}\right)f(s)\,ds\qquad t\in(a+\varepsilon,b-\varepsilon), 
    \]
and realize that it is infinitely many times differentiable. We may also write
    \[
    f_\varepsilon(t)=\int_{-1}^1\varphi(s)f(t-\varepsilon s)\,ds\qquad t\in(a+\varepsilon,b-\varepsilon).
    \]
    If $ f $ is continuous, then $ f_\epsilon $ is eventually well-defined and converges uniformly towards $ f $ on any compact subinterval of $ (a,b). $ In particular, for each $ t\in(a,b), $ the net $ f_\varepsilon(t) $ is well-defined for sufficiently small $ \varepsilon $ and converges to $ f(t) $ as $ \varepsilon $ tends to zero.

    Suppose now that $ f $ is $ n $-monotone in $ (0,\infty) $ for $ n\ge 2. $ We notice that $ f $ is continuous by Theorem~\ref{theorem: 2n-monotone function is n-concave}. It follows from the last integral representation that $ f_\epsilon $ is $ n $-monotone in the interval $ (\varepsilon,\infty) $ for $ \varepsilon>0. $ We realise that the restriction of $ f $ to any compact interval $ J $ in $ (0,\infty) $ is the uniform limit of a sequence of $ n $-monotone functions that are infinitely many times differentiable in a neighbourhood of $ J. $

    A similar statement is obtained for $ n $-convex functions defined in an open interval $ (a,b).  $ Notice that in this case the continuity  is immediate.

\section{Bendat and Sherman's theorem}

For a differentiable function $ f\colon I\to\mathbf R $ the (first) divided difference $ [t, s]_f $ for $ t,s\in I $ is defined by
\[
	[t, s]_f=\left\{
	\begin{array}{ll}\displaystyle
	\frac{f(t)-f(s)}{t-s}\quad &t\neq s\\[2ex]
	f'(t)& t=s,
	\end{array}\right.
\]
and the Löwner matrix $ L(\lambda_1,\dots,\lambda_n) $ is defined by setting
\[
L(\lambda_1,\dots,\lambda_n)=\Big([\lambda_i, \lambda_j]_f\Big)_{i,j=1}^n
\]
for $ \lambda_1,\dots,\lambda_n\in I. $ Notice that a Löwner matrix is linear in the function $ f. $

If $ f $ is twice continuously differentiable the second divided difference $ [t,s,r]_f $ for distinct numbers $ s,t,r\in I $ is defined by setting
\[
[t,s,r]_f=\frac{[t,s]_f-[s,r]_f}{t-r}
\]
and the definition is then extended by continuity to arbitrary numbers $ t,s,r\in I. $ Notice that in this way $ [t,t,t]_f=f''(t)/2. $

Divided differences are symmetric in the entries.

\begin{lemma}
Let $ f $ be a real function in $ C^1(I), $ where $ I $ is an open interval, and let $ x $ be an $ n\times n $ diagonal matrix with diagonal elements $ \lambda_1,\dots,\lambda_n\in I. $ The function $ t\to f(x+th) $ is  defined in a neighbourhood of zero for any hermitian $ n\times n $ matrix $ h=(h_{i,j})_{i,j=1}^n $ and
\[
\df{}{t} \bigl(f(x+th)\xi\mid\xi\bigr)\Big|_{t=0}=(h\circ L(\lambda_1,\dots,\lambda_n)\xi\mid\xi)\qquad\xi\in\mathbf C^n,
\]
where $ h\circ L(\lambda_1,\dots,\lambda_n) $ denotes the Hadamard (entry-wise) product of $ h $ and $ L(\lambda_1,\dots,\lambda_n). $
\end{lemma}

\begin{proof}
We first prove the lemma for a monomial $ f(t)=t^m, $ where $ m\ge 1 $ is an integer. The first divided difference
\[
\begin{array}{l}
\displaystyle [\lambda_i,\lambda_j]_f=\frac{\lambda_i^m-\lambda_j^m}{\lambda_i-\lambda_j}
=\lambda_i^{m-1}+\lambda_i^{m-2}\lambda_j+\cdots+\lambda_i\lambda_j^{m-2}+\lambda_j^{m-1}\\[3ex]
=\displaystyle\sum_{a+b=m-1}\lambda_i^a\lambda_j^b
\end{array}
\]
and this holds also for $ \lambda_i=\lambda_j. $ Therefore,
\[
\begin{array}{l}
\displaystyle (h\circ L(\lambda_1,\dots,\lambda_n)\xi\mid\xi)=\sum_{i=1}^n \bigl(h\circ L(\lambda_1,\dots,\lambda_n)\xi\bigr)_i\bar\xi_i\\[2ex]
=\displaystyle\sum_{i,j=1}^n h_{i,j} \sum_{a+b=m-1}\lambda_i^a\lambda_j^b\xi_j\bar\xi_i
=\sum_{a+b=m-1} (x^a h x^b\xi\mid\xi)
\end{array}
\]
which is the first order term i $ t $ of $ \bigl((x+th)^m\xi\mid\xi\bigr). $ By using linearity the statement of the lemma follows for arbitrary polynomials. The general case then follows by approximation.
\end{proof}

\begin{theorem}\label{characterization in terms of Loewner matrices}
Let $ f $ be a real function in $ C^1(I), $ where $ I $ is an open interval and take a natural number $ n\ge 1. $ Then $ f $ is $ n $-monotone if and only if the Löwner matrix $ L(\lambda_1,\dots,\lambda_n) $ is positive semi-definite for all sequences $ \lambda_1,\dots,\lambda_n\in I. $
\end{theorem}

\begin{proof}
It follows from classical analysis that $ f $ is $ n $-monotone if and only if 
\[
\df{}{t} \bigl(f(x+th)\xi\mid\xi\bigr)\Big|_{t=0}\ge 0
\]
for every hermitian $ x $ with spectrum in $ I, $ every positive semi-definite matrix $ h, $ and every $ \xi\in\mathbf C^n. $ We may now choose $ x $ as a diagonal matrix with diagonal elements $ \lambda_1,\dots,\lambda_n\in I. $ By choosing $ h $ as the positive semi-definite matrix with $ h_{i,j}=1 $ for $ i,j=1,\dots,n $ we realise that $ L(\lambda_1,\dots,\lambda_n)\ge 0 $ if $ f $ is $ n $-monotone. Since the Hadamard product of two semi-definite matrices is positive semi-definite (indeed, it is a principal submatrix of the tensor product), we realise that $ f $ is $ n $-monotone if all the Löwner matrices $  L(\lambda_1,\dots,\lambda_n)\ge 0 $ for arbitrary $ \lambda_1,\dots,\lambda_n\in I. $ 
\end{proof}

\begin{definition}
Take a function  $ f\in C^2(I), $ where $ I $ is an open interval and (not necessarily distinct) numbers $ \lambda_1,\dots,\lambda_n\in I. $ The associated Kraus~\cite{kn:kraus:1936, kn:hansen:1997:2} matrices $ H(1),\dots,H(n) $ are defined by setting
\[
H(p)=2\Big([\lambda_p,\lambda_i,\lambda_j]_f\Bigr)_{i,j=1}^n
\]
for $ p=1,\dots,n. $

\end{definition}
Notice that a Kraus matrix is linear in the function $ f. $

\begin{lemma}
Let $ f $ be a real function in $ C^2(I), $ where $ I $ is an open interval and let $ x $ be an $ n\times n $ diagonal matrix with diagonal elements $ \lambda_1,\dots,\lambda_n\in I. $ The function $ t\to f(x+th) $ is defined in a neighbourhood of zero for any hermitian $ n\times n $ matrix $ h=(h_{i,j})_{i,j=1}^n $  and
\[
\dto{}{t} \bigl(f(x+th)\xi\mid\xi\bigr)\Big|_{t=0}=\sum_{p=1}^n \bigl(H(p) \eta(p)\mid \eta(p)\bigr),
\]
where
\begin{enumerate}[(i)]

\item $ \xi=(\xi_1,\dots,\xi_n) $ is a vector in $ \mathbf C^n. $

\item $ H(1),\dots,H(n) $ are the Kraus matrices associated with $ f $ and $ \lambda_1,\dots,\lambda_n. $ 

\item $ \eta(p)=\bigl(\xi_1 h_{p,1},\dots,\xi_n h_{p,n}\bigr) $ for $ p=1,\dots,n. $

\end{enumerate}
\end{lemma}

\begin{proof}
We first prove the lemma for a monomial $ f(t)=t^m, $ where $ m\ge 2 $ is an integer. Since
\[
\begin{array}{l}
\displaystyle [\lambda_p,\lambda_i]_f-[\lambda_i,\lambda_j]_f\\[2ex]
\displaystyle=\lambda_p^{m-1}+\lambda_p^{m-2}\lambda_i+\cdots+\lambda_p\lambda_i^{m-2}+\lambda_i^{m-1}\\[1ex]
\hskip 10em-\bigl(\lambda_i^{m-1}+\lambda_i^{m-2}\lambda_j+\cdots+\lambda_i\lambda_j^{m-2}+\lambda_j^{m-1}\bigr)\\[2ex]
=\lambda_i^{m-2}(\lambda_p-\lambda_j)+\lambda_i^{m-3}(\lambda_p^2-\lambda_j^2)+\cdots+
\lambda_i(\lambda_p^{m-2}-\lambda_j^{m-2})+\lambda_p^{m-1}-\lambda_j^{m-1}
\end{array}
\]
the second divided difference
\[
\begin{array}{l}
\displaystyle[\lambda_p,\lambda_i,\lambda_j]_f=\frac{[\lambda_p,\lambda_i]_f-[\lambda_i,\lambda_j]_f}{\lambda_p-\lambda_j}\\[2ex]
=\lambda_i^{m-2}+\lambda_i^{m-3}(\lambda_p+\lambda_j)+\lambda_i^{m-4}(\lambda_p^2+\lambda_p\lambda_j+\lambda_j^2)+\cdots\\[1ex]
\hskip 8em +\,\lambda_i(\lambda_p^{m-3}+\lambda_p^{m-4}\lambda_j+\cdots+\lambda_p\lambda_j^{m-4}+\lambda_j^{m-3})\\[1ex]
\hskip 10em+\,(\lambda_p^{m-2}+\lambda_p^{m-3}\lambda_j+\cdots+\lambda_p\lambda_j^{m-3}+\lambda_j^{m-2})\\[2ex]
=\displaystyle\sum_{a+b+c=m-2}\lambda_p^a\lambda_i^b\lambda_j^c\,,
\end{array}
\]
where summation limits should be properly interpreted, and the case $ \lambda_p=\lambda_j $ is handled separately. We then obtain
\[
\begin{array}{l}
\displaystyle\sum_{p=1}^n \big(H(p)\eta(p)\mid\eta(p)\bigr)=\sum_{p,i,j=1}^n 2[\lambda_p,\lambda_i,\lambda_j]_f \xi_j h_{p,j} \bar\xi_i \bar h_{p,i}\\[2ex]
=\displaystyle 2\sum_{i,j,p=1}^n\sum_{a+b+c=m-2}\lambda_p^a\lambda_i^b\lambda_j^c \xi_j h_{p,j}\bar\xi_i \bar h_{p,i}=
2\sum_{a+b+c=m-2} (x^b h x^a h x^c\xi\mid\xi)
\end{array}
\]
which is the second order term in $ t $ of $ \bigl((x+th)^{m}\xi\mid\xi\bigr). $ By using linearity the statement of the lemma follows for arbitrary polynomials. The general case then follows by approximation.
\end{proof}

\begin{theorem}\label{convexity in terms of Kraus matrices}
Let $ f $ be a real function in $ C^2(I), $ where $ I $ is an open interval. Then $ f $ is $ n $-convex if and only if the Kraus matrices associated with $ f $ and any choice of $ \lambda_1,\dots,\lambda_n\in I $ are positive semi-definite.
\end{theorem}

\begin{proof}
The sufficiency of the conditions is obvious from the above lemma.
Assume now that $ f $ is $ n $-convex and choose $ \lambda_1,\dots,\lambda_n\in I. $

Take a fixed $ p=1,\dots,n $ and a fixed vector $ \eta\in\mathbf C^n. $ To a given $ \varepsilon>0 $ we choose a vector $ \xi $  by setting $ \xi_i=\varepsilon^{-1} $ for $ i\ne p $ and $ \xi_p=1. $  We then choose a vector $ a $ by setting
\[
a_i=\frac{\eta_i}{\xi_i}=\left\{\begin{array}{ll}
                                                                      \varepsilon\eta_i    &i\ne p\\[0.5ex]
                                                                      \eta_p                      &i=p.
                                                                      \end{array}\right.
\]
We finally choose a self-adjoint (actually a positive semi-definite) matrix $ h $ by setting $ h_{i,j}=\bar a_i a_j $ for $ i,j=1,\dots,n $ and calculate
\[
\eta(q)_i=\xi_i h_{q,i}=\xi_i \bar a_q a_i\qquad q=1,\dots,n.
\]
With these choices we obtain
\[
\eta(p)=\bar\eta_p\,\eta\qquad\text{and}\qquad\eta(q)=\varepsilon\bar\eta_q\,\eta\quad\text{for}\quad q\ne p.
\]
Therefore,
\[
\dto{}{t} \bigl(f(x+th)\xi\mid\xi\bigr)\Big|_{t=0}=|\eta_p|^2 \bigl(H(p)\eta\mid\eta)+ \varepsilon^2\sum_{q\ne p}^n |\eta_q|^2  \bigl(H(q)\eta\mid\eta\bigr)
\]
is non-negative, and since $ \eta $ is a fixed vector we obtain 
\[
|\eta_p|^2 \bigl(H(p)\eta\mid\eta)\ge 0
\]
by letting $ \varepsilon $ tend to zero. In particular, $ \bigl(H(p)\eta\mid\eta)\ge 0 $ for all vectors $ \eta\in\mathbf C^n $ with $ \eta_p\ne 0. $  By continuity we finally realize that $ H(p) $ is positive semi-definite.
\end{proof}

\begin{theorem}[Bendat and Sherman]\label{theorem: Bendat and Sherman}
Let $ f $ be a real function in $ C^2(I),  $ where $ I $ is an open interval. Then $ f $ is operator convex if and only if the function
\[
g(t)=\left\{\begin{array}{ll}
                \displaystyle\frac{f(t)-f(t_0)}{t-t_0}\quad &t\ne t_0\\[2.5ex]
                 f'(t_0)                                                           &t=t_0
                 \end{array}\right.
\]
is operator monotone for each $ t_0\in I. $
\end{theorem}

\begin{proof}
Using the symmetry of divided differences we realise that
\[
[\lambda_i,\lambda_j]_g=[t_0,\lambda_i,\lambda_j]_f\qquad i,j=1,\dots,n.
\]
If $ f $ is $ (n+1) $-convex then $ g $ is $ n $-monotone for each $ t_0\in I $ by Theorem~\ref{characterization in terms of Loewner matrices} and Theorem~\ref{convexity in terms of Kraus matrices}. Conversely,  if $ g $ is $ n $-monotone for each $ t_0\in I $ then $ f $ is $ n $-convex.
\end{proof}

\subsection{Further preparations}

\begin{theorem}[Bendat and Sherman]\label{theorem: Bendat and Sherman, general case}
Let $ f $ be an operator convex function defined in the positive half-line. Then $ f $ is differentiable, and the function
\[
g(t)=\left\{\begin{array}{ll}
                \displaystyle\frac{f(t)-f(t_0)}{t-t_0}\quad &t\ne t_0\\[2.5ex]
                 f'(t_0)                                                           &t=t_0
                 \end{array}\right.
\]
is operator monotone for each $ t_0>0. $
\end{theorem}

\begin{proof}

Suppose $ f $ is operator convex, thus in particularly continuous. Using regularization (with $ \varepsilon < t_0) $ we obtain $ f $ as the point-wise limit, for $ \varepsilon\to 0, $ of a sequence $ (f_{\varepsilon})_{\varepsilon>0} $ of infinitely differentiable operator convex functions. The functions
\[
g_\varepsilon(t)=\left\{\begin{array}{ll}
                \displaystyle\frac{f_\varepsilon(t)-f_\varepsilon(t_0)}{t-t_0}\quad &t\ne t_0\\[3ex]
                 \displaystyle f_\varepsilon'(t_0)\quad                                                           &t=t_0
                 \end{array}\right.
\]
are operator monotone in $ (\epsilon,\infty) $ by Theorem~\ref{theorem: Bendat and Sherman}. In addition, $ g_\varepsilon(t)\to g(t) $ for $ t\ne t_0. $ Since $ f $ is convex the set of derivatives $ \{f'_\varepsilon(t_0)\} $ is bounded for small $ \varepsilon<t_0. $ A subsequence of $ (g_\varepsilon)_{\epsilon>0} $ therefore converges towards an operator monotone function which is continuous according to Theorem~\ref{theorem: 2n-monotone function is n-concave}. But then $ f $ is differentiable in $ t_0 $ and we conclude that $ f'(t_0)=\lim_{\varepsilon\to 0} f_\varepsilon'(t_0). $ 
\end{proof}

In the above proof we also learn that $ f'_\varepsilon(t)\to f'(t) $ for every $ t\in (0,\infty), $ where $ f_\varepsilon $ is the regularization of $ f. $ In connection with Corollary~\ref{monotonicity and concavity} we obtain

 \begin{corollary}\label{corollary: automatic differentiability}
 An operator monotone or operator convex function $ f $ defined in the positive half-line is automatically differentiable and $ f'_\varepsilon(t)\to f'(t) $ for every $ t\in(0,\infty), $ where $ f_\epsilon $ is the regularization of $ f. $
  \end{corollary}
 
 By applying regularization of $ f  $ and then appealing to Theorem~\ref{characterization in terms of Loewner matrices} and Corollary~\ref{corollary: automatic differentiability}  we obtain:
 
 \begin{corollary}\label{corollary: positivity of Loewner matrices}
Let $ f $ be an operator monotone function defined in the positive half-line. The Löwner matrices $  L(\lambda_1,\dots,\lambda_n) $ associated with $ f $ are well-defined and positive semi-definite for arbitrary  $ \lambda_1,\dots,\lambda_n\in I. $ 
\end{corollary}

 \begin{corollary}\label{the derivative of a non-constant 2-monotone function cannot have a zero}
Let $ f $ be an operator monotone function defined in the open half-line. If the derivative $ f'(t)=0 $ in any point $ t>0, $ then $ f $ is a constant function.
\end{corollary}

\begin{proof}
The Löwner matrix
\[
L(t,s)=\begin{pmatrix}
            f'(t) & [t,s]_f\\[0.5ex]
            [s,t]_f & f'(s)
            \end{pmatrix}\qquad t\ne s
\]
is well-defined and positive semi-definite by Corollar~\ref{corollary: positivity of Loewner matrices}, thus
\[
f'(t)f'(s)\ge\left(\frac{f(t)-f(s)}{t-s}\right)^2.
\]
If $ f'(t)=0, $ then necessarily $ f(s)=f(t) $ for every $ s>0. $
\end{proof}

Notice that we for this result only need $ 2 $-monotonicity of $ f, $ cf.~\cite{kn:hansen:2007:1}.

\section{The fast track to Löwner's theorem}

 \begin{lemma}\label{main lemma on involutions}
 Let $ f\colon(0,\infty)\to(0,\infty) $ be an operator monotone function. The function $ t\to t^{-1}f(t) $ is operator monotone decreasing. 
 \end{lemma}
  
 \begin{proof}
 For $ \varepsilon>0 $ the function $ f_{\varepsilon}(t)= f(t+\varepsilon) $ is defined in the open set $ (-\varepsilon,\infty) $ containing zero. Since $ f $ and hence $ f_{\varepsilon} $ are operator monotone and therefore operator concave by Corollary~\ref{operator monotonicity implies operator concavity} we may use Theorem~\ref{theorem: Bendat and Sherman, general case} (Bendat and Sherman) to obtain that the function
 \[
 t\to\frac{f_{\varepsilon}(t)-f_{\varepsilon}(0)}{t-0}=\frac{f(t+\varepsilon)-f(\varepsilon)}{t}
 \]
 is operator monotone decreasing.  By using  $ f(\varepsilon)>0 $ and the identity
 \[
 \frac{f(t+\varepsilon)}{t}=\frac{f(t+\varepsilon)-f(\varepsilon)}{t}+\frac{f(\varepsilon)}{t}
  \]
  we realize that the function $ t\to t^{-1} f(t+\varepsilon) $ is operator monotone decreasing when restricted to the positive half-line. The result now follows by letting $ \varepsilon $ tend to zero.\hfill $ \Box $
 \end{proof}

 \begin{corollary}\label{the two involution of positive oper mon func}
 Let $ f\colon(0,\infty)\to(0,\infty) $ be an operator monotone function. The functions
 \[
 f^\sharp(t)=t f(t)^{-1}\quad\text{and}\quad f^*(t)=t f(t^{-1})
 \]
 are operator monotone in the positive half-line.
\end{corollary}

\begin{proof}
Since $ t\to f^\sharp(t)^{-1}=t^{-1}f(t) $ is operator monotone decreasing by the above lemma it follows that $ f^\sharp $ is operator monotone (increasing). The second assertion follows from the same argument by first replacing $ f $ with the operator monotone function $ t\to f(t^{-1})^{-1}. $ 
\end{proof}

The corollary states that the mappings $ f\to f^\sharp $ and $ f\to f^* $ are involutions of the set of positive operator monotone functions defined in the positive half-line.  

\begin{lemma}\label{bound for positive operator monotone function}
We have the bound $ f(t)\le t+1 $ for any positive operator monotone function $ f $ defined in the positive half-line with $ f(1)=1. $
\end{lemma}

\begin{proof}
Since $ f $ is increasing we obviously have
\[
f(t)\le f(1)=1\le t+1\qquad\text{for }\, 0<t\le 1.
\] 
We also notice that $ f $ is concave by Theorem~\ref{theorem: 2n-monotone function is n-concave}.
It follows, for $ t>1, $ that $ f(t) $ is bounded by the continuation of the chord between $ (0, \lim_{\varepsilon\to 0}f(\varepsilon)) $ and $ (1,f(1))=(1,1). $ But the continuation of this chord is bounded by $ t+1. $
\end{proof}

Let $ \mathcal P $ denote the set of positive operator monotone functions
defined in the positive half-line and consider the convex set
\[
\mathcal P_0=\{f\in\mathcal P\mid f(1)=1\}.
\]
We equip $ \mathcal P_0 $ with the topology of point-wise convergence and realize, by the preceding lemma, that $ \mathcal P_0 $ is compact in this topology.

\begin{theorem}\label{reduction to positive functions}
Let $ f\colon(0,\infty)\to\mathbf R $ be a non-constant operator monotone function. Then $ f $ can be written on the form
\[
f(t)=f(1)+f'(1)\frac{t-1}{t} (\T f)(t)\qquad t>0,
\]
where $ \T f\in \mathcal P_0   $ is given by 
\[
(\T f)(t)=\frac{t}{f'(1)}\cdot \left\{\begin{array}{ll}
                                            \displaystyle \frac{f(t)-1}{t-1}\qquad &t\ne 1\\[2ex]
                                            f'(1)                                                        &t=1.
                                            \end{array}\right.
\]
\end{theorem}

Notice that $ f'(1)>0 $ by Corollary~\ref{the derivative of a non-constant 2-monotone function cannot have a zero} since $ f $ is non-constant.

\begin{proof}
The function
\[
h_1(t)=\frac{1}{f'(1)}\cdot\frac{f(t)-f(1)}{t-1}
\]
is positive since $ f $ is strictly increasing, and $ h_1(1)=1. $ Since $ f $ is operator monotone and thus operator concave the function $ h_1 $ is operator monotone decreasing by Theorem~\ref{theorem: Bendat and Sherman, general case}.  By composing with the operator monotone decreasing function $ t\to t^{-1} $ we obtain that
\[
h_2(t)=h_1(t^{-1})=\frac{1}{f'(1)}\cdot\frac{f(t^{-1})-f(1)}{t^{-1}-1}
\]
is positive and operator monotone with $ h_2(1)=1. $ By applying the involution $ h_2\to h_2^* $ we finally obtain that the function
\[
(\T f) (t)=h_2^*(t)=t h_2(t^{-1})= \frac{t}{f'(1)}\cdot\frac{f(t)-f(1)}{t-1}
\]
is operator monotone by Corollary~\ref{the two involution of positive oper mon func}. It is also positive and   $ (\T f)(1)=1. $  The assertion now follows by solving the equation for $f. $
\end{proof}

\begin{lemma}\label{The two invariant operations}
The involution $ f\to f^* $ maps $ \mathcal P_0 $ into itself, and the operation $ f\to\T f $ maps the non-constant functions in $ \mathcal P_0 $ into $ \mathcal P_0. $ 
\end{lemma}

\begin{proof}
Follows immediately from Corollary~\ref{the two involution of positive oper mon func} and Theorem~\ref{reduction to positive functions}.
\end{proof}

\begin{lemma}
The sum of the derivatives
\[
\left.\df{}{t}f(t)\right|_{t=1}+\left.\df{}{t}f^*(t)\right|_{t=1}=1
\]
for any $ f\in \mathcal P_0. $
\end{lemma}

\begin{proof}The assertion follows from the calculation
\[
\frac{f(t)-1}{t-1}+\frac{f^*(t^{-1})-1}{t^{-1}-1}=1
\]
by letting $ t $ tend to $ 1. $
\end{proof}

Both $ f $ and $ f^* $ are increasing functions. By Corollary~\ref{the derivative of a non-constant 2-monotone function cannot have a zero} we therefore obtain:

\begin{corollary}
The derivative of $ f $ satisfies
\[
0<f'(1)<1
\]
for any function $ f\in\mathcal P_0 $ different from the constant function $ t\to1 $ or the identity function
$ t\to t. $

\end{corollary}

\begin{lemma}\label{formula for the extreme points}
An extreme point $ f $ in $ \mathcal P_0 $ is necessarily of the form
\[
f(t)=\frac{t}{f'(1)+(1-f'(1))t}\qquad t>0.
\]
\end{lemma}

\begin{proof}
Take first a function $ f\in\mathcal P_0 $ which is neither the constant function $ t\to 1 $
nor the identity function $ t\to t, $ thus $ \lambda=f'(1)\in(0,1) $ by the above corollary. An elementary calculation shows that
\begin{equation}
\lambda\T{f} + (1-\lambda) (\T{f^*})^*=f.
\end{equation}
Indeed,
\[
\lambda(\T{f})(t)=t\,\frac{f(t)-1}{t-1}\qquad t\ne 1
\]
and
\[
(1-\lambda) (\T{f^*})^*(t)=(1-\lambda)t(\T{f^*})(t^{-1})=\frac{f^*(t^{-1})-1}{t^{-1}-1}= \frac{f(t)-t}{1-t}
\]
from which the assertion follows. Consequently, if $ f $ is an extreme point in $ \mathcal P_0 $ then
$ \T{f}=f $ or
\[
\frac{t}{\lambda}\cdot\frac{f(t)-1}{t-1}=f(t)\qquad t>0
\]
from which it follows that
\[
f(t)=\frac{t}{\lambda+(1-\lambda)t}\qquad t>0.
\]
 Finally, the two functions we left out may also be written in this way. Indeed, the constant function $ t\to 1 $ appears in the formula by setting $ \lambda=0 $ while the identity function $ t\to t $ appears by setting $ \lambda=1. $
\end{proof}

\begin{theorem}\label{theorem: formula for oper mon func with measure on [0,1]}
Let $ f $ be a positive operator monotone function defined in the positive half-line.
There is a bounded positive measure $ \mu $ on the closed interval $ [0,1] $ such that
\[
f(t)=\int_0^1 \frac{t}{\lambda+(1-\lambda)t}\,d\mu(\lambda)\qquad t>0.
\]
Conversely, any function given on this form is operator monotone. The measure $ \mu $ is a probability measure if and only if $ f(1)=1. $
\end{theorem}

\begin{proof}
We noticed that $ \mathcal P_0 $ is convex and compact in the topology of point-wise convergence of functions.
Therefore, by Krein-Milman's theorem, it is generated by its extreme points $ \mathit{Ext}(\mathcal P_0) $ in the sense that $ \mathcal P_0 $ is the closure
\[
\mathcal P_0=\mathit{\overline{conv}(Ext}(\mathcal P_0))
\]
of the convex hull of $ \mathit{Ext}(\mathcal P_0). $ By Lemma~\ref{formula for the extreme points} the convex hull of $ \mathit{Ext}(\mathcal P_0) $ consists of functions of the form
\begin{equation}\label{formula with a discrete measure}
f(t)=\int_0^1 \frac{t}{\lambda+(1-\lambda)t}\,d\mu(\lambda)\qquad t>0,
\end{equation}
where $ \mu $ is a discrete probability measure on $ [0,1]. $  A function $ f $ in $ \mathcal P_0 $
is therefore the limit of a net of functions $ (f_j)_{j\in J} $ written on the form (\ref{formula with a discrete measure}) in terms of discrete probability measures $ (\mu_j)_{j\in J}. $ Since the set of probability measures on $ [0,1] $ is compact in the weak topology there exists an accumulation measure $ \mu $ such that $ f $ is expressed as in the statement of the theorem. 
\end{proof}

Notice that a possible atom in zero of the measure $ \mu $ in the above theorem contributes with the constant term $ \mu\{0\}  $ in the integral. A possible atom in $ 1 $ contributes with the term $ \mu\{1\} t. $

A brief outline of the theory presented in Theorem~\ref{reduction to positive functions}, Lemma~\ref{formula for the extreme points}  and Theorem~\ref{theorem: formula for oper mon func with measure on [0,1]}  was given in the authors' PhD thesis~\cite[Page 12-13]{kn:hansen:1983:2}.

It is illuminating to consider the linear mapping
\[
\Lambda(f)(t)=\left\{\begin{array}{ll}
                                            \displaystyle t\frac{f(t)-f(1)}{t-1}\qquad &t\ne 1\\[2ex]
                                            f'(1)                                                             &t=1
                                            \end{array}\right.
\]
defined for differentiable functions $ f\colon(0,\infty)\to\mathbf R. $ It is closely related to the non-linear transformation $ T $ introduced in Theorem~\ref{reduction to positive functions}. Indeed, 
\[
\Lambda(f)=f'(1) \T{f}\qquad\text{for}\quad f\in\mathcal P_0
\]
and $ \Lambda $ is thus a transformation of $ \mathcal P. $
However, we cannot replace $ T $ by $ \Lambda $ in the proof of Lemma~\ref{formula for the extreme points} since $ \Lambda $ does not map $ \mathcal P_0 $ into itself, and we cannot alternatively work directly with $ \mathcal P $ since $ \mathcal P $ is not compact.

\begin{theorem}
The measure $ \mu $ appearing in Theorem~\ref{theorem: formula for oper mon func with measure on [0,1]} is uniquely defined by the operator monotone function $ f\in\mathcal P. $
\end{theorem}

\begin{proof}
The action of $ \Lambda $ on functions in $ \mathcal P $ is calculated by noticing that
\[
\Lambda\left(\frac{t}{\lambda+(1-\lambda)t}\right)=\frac{\lambda t}{\lambda+(1-\lambda)t}
\]
and thus
\[
p(\Lambda)\left(\frac{t}{\lambda+(1-\lambda)t}\right)= \frac{p(\lambda) t}{\lambda+(1-\lambda)t}
\]
for any polynomial $ p. $ For a function $ f\in\mathcal P $ we thus have
\[
p(\Lambda)(f)=\int_0^1 \frac{t}{\lambda+(1-\lambda)t}\, p(\lambda)\,d\mu(\lambda),
\]
where $ \mu $ is the representing measure for $ f. $ This identity recovers the measure $ \mu $ from $ f $ by using Weierstrauss's polynomial approximation theorem.
\end{proof}

\section{Other integral representations}

\begin{corollary}\label{integral formula for positive operator monotone functions}
Let $ f $ be a positive operator monotone function defined in the positive half-line.
There is a bounded positive measure $ \mu $ on the closed extended half-line $ [0,\infty] $ such that
\[
f(t)=\int_0^\infty \frac{t(1+\lambda)}{t+\lambda}\, d\mu(\lambda)\qquad t>0.
\]
Conversely, any function given on this form is operator monotone. The measure $ \mu $ is a probability measure if and only if $ f(1)=1. $
\end{corollary}

\begin{proof}
The assertion follows from the previous theorem by applying the transformation 
\[
\lambda\to\alpha=\lambda(1-\lambda)^{-1}
\]
which maps
the closed interval $ [0,1] $ onto the closed extended half-line $ [0,\infty], $ and by noticing the identity
\[
\frac{t}{\lambda+(1-\lambda)t}= \frac{t(1-\lambda)^{-1}}{\lambda(1-\lambda)^{-1}+t}=\frac{t(1+\alpha)}{t+\alpha}
\]
which is valid also in the end points of the two intervals.
\end{proof}

We are finally able to give an integral formula for the  operator monotone functions defined in the positive half-line. There are various ways of doing so, but the following formula establishes the connection between operator monotone functions and the theory of Pick functions \cite{kn:donoghue:1974} in complex analysis.

\begin{theorem}\label{formula for operator monotone function}
Let $ f\colon(0,\infty)\to\mathbf R $ be an operator monotone function. There exists a positive measure $ \nu $ on the closed positive half-line $ [0,\infty) $ with $ \int (1+\lambda^2)^{-1}\, d\nu(\lambda)<\infty $ such that
\[
f(t)=\alpha t+\beta+\int_0^\infty\left(\frac{\lambda}{1+\lambda^2} - \frac{1}{t+\lambda}\right)\, d\nu(\lambda)\qquad t>0,
\]
where $ \alpha\ge 0 $ and $ \beta\in\mathbf R. $ Conversely, any function given on this form is operator monotone.
\end{theorem}

\begin{proof}
We first use Theorem~\ref{reduction to positive functions} to write $ f $ on the form
\[
f(t)=f(1)+f'(1)\frac{t-1}{t} (\T f)(t)\qquad t>0,
\]
where $ \T f  $ is a positive and normalized operator monotone function. We can then apply Corollary~\ref{integral formula for positive operator monotone functions} to obtain a probability measure $ \mu $ on the closed extended half-line $ [0,\infty] $ such that
\[
f(t)=f(1)+f'(1)\frac{t-1}{t}\int_0^\infty \frac{t(1+\lambda)}{t+\lambda}\, d\mu(\lambda)\qquad t>0.
\]
We explicitly remove a possible atom in $ \infty $ to obtain
\[
f(t)=f(1)+f'(1)\mu(\{\infty\}) (t-1)+ f'(1)\int_0^\infty \frac{(t-1)(1+\lambda)}{t+\lambda}\, d\tilde\mu(\lambda),
\]
where $ \tilde\mu $ is a positive finite measure on the closed half-line $ [0,\infty). $ We then make use of the identity
\[
 \frac{(t-1)(1+\lambda)}{t+\lambda}=(1+\lambda)^2\left(\frac{\lambda}{1+\lambda^2}-\frac{1}{t+\lambda}\right)
 +\frac{1-\lambda^2}{1+\lambda^2}
\]
to obtain
\[
f(t)=\alpha t+\beta+ f'(1)\int_0^\infty (1+\lambda)^2\left(\frac{\lambda}{1+\lambda^2}-\frac{1}{t+\lambda}\right)\, d\tilde\mu(\lambda),
\]
where $ \alpha=f'(1)\mu(\{\infty\})\ge 0 $ and 
\[
\beta=f(1)-\mu(\{\infty\})f'(1)+f'(1)\int_0^\infty \frac{1-\lambda^2}{1+\lambda^2}\, d\tilde\mu(\lambda)
\]
is finite since the integrand is bounded between $ -1 $ and $ 1. $ The assertion now follows by setting $ d\nu(\lambda)=f'(1)(1+\lambda)^2\, d\tilde\mu(\lambda) $ and noticing that
\[
1\le (1+\lambda)^2/(1+\lambda^2)\le 2
\]
for $ 0\le\lambda<\infty. $
\end{proof}

\begin{remark}
The unicity of the representing measure $ \mu $ in Theorem~\ref{theorem: formula for oper mon func with measure on [0,1]} readily implies unicity of the representing measures in Corollary~\ref{integral formula for positive operator monotone functions} and Theorem~\ref{formula for operator monotone function}.
\end{remark}

\subsection{Löwner's theorem}

We learn from the integral expression in the previous theorem that an operator monotone function $ f $ defined in the positive half-line can be continued to an analytic function defined in $ \mathbf C\backslash (-\infty,0]. $ Since the imaginary part
\[
\Im\left(-\frac{1}{z+\lambda}\right)=\frac{\Im z}{|z+\lambda|^2}
\]
we also learn that the analytic continuation of $ f $ to the complex upper half-plane has non-negative imaginary part. In fact, the imaginary part of the continuation is positive if  $ f $ is not  constant.

\begin{theorem}[Löwner]
Let $ f:I\to \mathbf R $ be a function defined in an open interval which is either finite $ I=(a,b) $ or infinite of the form $ (a,\infty). $ Then $ f $ is operator monotone if and only if it allows an analytic continuation to the upper half-plane with non-negative imaginary part.
\end{theorem}

\begin{proof}
The case where $ I $ is the positive half-line follows from Theorem~\ref{formula for operator monotone function} and from the Theory of Pick functions \cite{kn:donoghue:1974}, and the case $ I=(a,\infty) $ then follows by a simple translation. The remaining cases may be similarly reduced to the case $ I=(0,1). $
The function,
\[
h(t)=\frac{t}{t+1}\qquad t>0,
\]
is a bijection between $ (0,\infty) $ and the interval $ (0,1). $ It is operator monotone, and the inverse function,
\[
h^{-1}(t)=\frac{t}{1-t}=\frac{1}{t^{-1}-1}\qquad 0<t<1,
\]
is also operator monotone. Both functions have analytic continuations which map the complex upper half-plane into itself. Composition with $ h $ therefore establishes a bijection between the operator monotone functions defined in the two intervals $ (0,\infty) $ and  $ (0,1). $ It also establishes a bijection between the functions defined in each of the two intervals, that allow an analytic continuation into the complex upper half-plane with non-negative imaginary part.
\end{proof}

\subsection{The representing measure}

\begin{theorem}\label{calculation of the associated measure}
Let $ f\colon(0,\infty)\to\mathbf R $ be an operator monotone function, and let $ \nu $ be the representing measure as given in Theorem~\ref{formula for operator monotone function}. Let $ \tilde\nu $ be the measure obtained from $ \nu $ by removing a possible atom in zero. Then
\[
\lim_{\varepsilon\to 0}\frac{1}{\pi}\int_0^\infty \Im f(-t+i\varepsilon) g(t)\, dt=\frac{g(0)}{2}\nu(\{0\})+\int_0^\infty g(\lambda)\, d\tilde\nu(\lambda)
\]
for every continuous, bounded and integrable function $ g $ defined in $ [0,\infty). $
\end{theorem}

\begin{proof}
By applying Theorem~\ref{formula for operator monotone function} we obtain
\[
\begin{array}{rl}
I_\varepsilon&=\displaystyle\frac{1}{\pi}\int_0^\infty \Im f(-t+i\varepsilon) g(t)\, dt\\[3ex]
&=\displaystyle\frac{1}{\pi}\int_0^\infty\left(\varepsilon\alpha + \int_0^\infty\frac{\varepsilon}{(\lambda-t)^2+\varepsilon^2}\, d\nu(\lambda) \right) g(t)\, dt. 
\end{array}
\]
By Fubini's theorem we may then write
\[
I_\varepsilon=\frac{\varepsilon\alpha}{\pi} \int_0^\infty g(t)\, dt +\frac{1}{\pi}\int_0^\infty\int_0^\infty\frac{\varepsilon}{(\lambda-t)^2+\varepsilon^2} g(t)\, dt\, d\nu(\lambda).
\]
Since 
\[
\frac{1}{\pi}\int_{-\infty}^\infty\frac{\varepsilon}{(\lambda-t)^2+\varepsilon^2}\, dt =1,
\]
we obtain by Lebesgue's convergence theorem
\[
\lim_{\varepsilon\to 0}\frac{1}{\pi}\int_0^\infty\frac{\varepsilon}{(\lambda-t)^2+\varepsilon^2} g(t)\, dt=g(\lambda)\qquad\text{for}\quad\lambda>0.
\]
For $ \lambda=0 $ we only obtain $ g(0)/2\,. $ 
\end{proof}

\begin{acknowledgement*}
The author is indebted to the referees for a number of useful suggestions.
\end{acknowledgement*}

{\small

%\bibliographystyle{plain}
 %\bibliography{mathharv}

\vfill

      \noindent Frank Hansen: Institute for International Education, Tohoku University, Japan. Email: frank.hansen@m.tohoku.ac.jp.
      }

\end{document}